\documentclass[times, 12pt, reqno]{amsart}

\usepackage{amsmath,amsfonts,amssymb,latexsym}
\usepackage{dsfont}
\usepackage[colorlinks=true,breaklinks=true]{hyperref}
\usepackage{amsbsy}
\usepackage{eufrak}
\usepackage{xcolor}
\usepackage{amsopn}
\usepackage{amsthm}
\usepackage[english]{babel}
\usepackage[applemac]{inputenc}
\usepackage[T1]{fontenc}
\usepackage{graphicx}
\usepackage{times}
\usepackage{enumerate}
\usepackage{stackrel}

\DeclareMathOperator{\er}{\mathds{E}}
\DeclareMathOperator{\pr}{\mathbb{P}}
\DeclareMathOperator{\du}{{\rm d}\it u}

\DeclareMathOperator{\dx}{{\rm d}\it x}
\newcommand{\simdis}{\stackrel{d}{=}}
\newcommand{\limdis}{\stackrel{d}{\longrightarrow}}
\newcommand{\II}{{\rm  1~\hspace{-1.1ex}l}}

\newcommand{\siminfini}{\stackbin[+\infty]{}{\sim}}

\newcommand{\TT}{\mathds{T}}
\newcommand{\lao}{{\lambda_{_0}}}
\newcommand{\ii}{{\rm i}}

\newtheorem{theorem}{Theorem}
\newtheorem{definition}{Definition}
\newtheorem{remark}{Remark}
\newtheorem{corollary}{Corollary}
\newtheorem{lemma}{Lemma}
\newtheorem{example}{Example}
\newtheorem{proposition}{Proposition}

\newcommand\rsetminus{\mathbin{\mathpalette\rsetminusaux\relax}}
\newcommand\rsetminusaux[2]{\mspace{-4mu}
  \raisebox{\rsmraise{#1}\depth}{\rotatebox[origin=c]{-20}{$#1\smallsetminus$}}
 \mspace{-4mu}
}
\newcommand\rsmraise[1]{%
  \ifx#1\displaystyle .8\else
    \ifx#1\textstyle .8\else
      \ifx#1\scriptstyle .6\else
        .45%
      \fi
    \fi
  \fi}

\title[Mellin transform, size Biasing and stationary excess operator]{Refined properties for the Mellin transform, with applications to convergence of families obtained by biasing or by the stationary excess operator}

\author[W. al Jedidi, F. Bouzeffour, N. Harthi]{Wissem al Jedidi$^{\diamond,\star}$, Fethi Bouzaffour$^{\diamond \diamond}$, Nouf Harthi $^{\diamond }$}
\address{\rm  $^{\star}$  Universit\'e de Tunis El Manar, Facult\'e des Sciences de Tunis, LR11ES11 Laboratoire d'Analyse Math\'ematiques et Applications, 2092, Tunis, Tunisie.}
\address{\rm  $^{\diamond}$  Department of Statistics \& OR, King Saud University, P.O. Box 2455, Riyadh 11451, Saudi Arabia  \\{\em Email} : {\tt wissem\_jedidi@yahoo.fr}  {\em Email} : {\tt nharthi@yahoo.com}}
\address{\rm $^{\diamond \diamond}$  King Saud Univ, Dept Math, Coll Sci, Riyadh 11451, Saudi Arabia  {\em Email} : {\tt fbouzaffour@ksu.edu.sa}}
\date{\today}
\begin{document}
\maketitle
\begin{abstract}  We first provide some properties of the Mellin transform of nonnegative random variables, such that monotonicity, injectivity and effect of size biasing. Convergence of Mellin transforms is also entirely formalized through convergence in distribution and uniform integrability. As an application, we study  a problem raised by   Harkness and Shantaram (1969) who obtained, under sufficient conditions, a limit theorem for sequences of nonnegative random variables build with the iterated stationary excess operator. We reformulate this problem through the concept of multiply monotone functions and through  the convergence of the families build by the continuous time version of the iterated stationary excess operator and also by size biasing. The latter allows us to  show that in our context, continuous time convergence is equivalent to discrete time convergence, that the conditions of Harkness and Shantaram  are actually necessary and that the only possible limits in distribution are  mixture of exponential with log-normal distributions.
\end{abstract}

\section{Introduction and main results}
The Mellin transform of a nonnegative random variable $X$ is defined by $$\mathcal{M}_X(\lambda)=\er[X^\lambda],\quad \mbox{for $\lambda$ in some domain of definition in the complex plane},$$
and  could be also interpreted as the moment generating function of $\log X$. We denote  $\mathcal{D}_X$ the domain of definition of $\mathcal{M}_X$ restricted on the real line.

Laplace transform of nonnegative random variables and characteristic functions of real-values random variables are always well defined, respectively on the half real line and on the real line, and they entirely characterize the distribution. At the contrary, the Mellin transform could have problems of definition and for this reason formalization of  its injectivity is not straightforward. Nevertheless, injectivity on the Mellin transform seems to be commonly admitted in the literature and used  without a precise reference. In Chapter VI of Widder's book \cite{widder},  Theorem 6a p. 243, it is stated that {\it if the Mellin transforms  of two  nonnegative random variables $X$ and $Y$ are well defined on some strip $\alpha< Re(z) < \beta$, and equal there, then $X\simdis Y$}. We improve this result by showing that the same conclusion holds if  {\it the strip is replaced by an interval}.

Mellin transform works well with size biased distributions, i.e. for distributions of random variables weighted by a power function:
$$ X_{(t)} :\simdis \frac{x^t}{\er[X^t]} \mathbb{P}(X\in dx), \quad t \in \mathcal{D}_X .$$
The distribution of the random variable $X_{(t)}$ is called the size biased law of transform of order $t$ of $X$. It is well know that   $X_{(t)}$ is stochastically bigger than $X$  when  $t>0$ and we were interested by deriving new properties. In the dissertation of the third author \cite{Har},  some convexity properties of the Mellin  transform of $X_{(t)}$ were required and we were addressed to the monotonicity property of the  $t\mapsto \mathcal{M}_X(\lambda+t)/\mathcal{M}_X(t), \, t\geq 0$ for fixed $\lambda>0$. This made as aware of the works of Harkness  and Shantaram \cite{hs1} who proved that the last function is non-decreasing in $t \in \mathds{\mathds{N}}$ when $\lambda=1$. We prove that is also true in $\, t\in [0,\infty)$ for any $\lambda>0$.

Harkness  and Shantaram \cite{hs1} were actually motivated by a limit theorem that we explain as follows. Let $X$ a nonnegative random variable having moments of all orders. The stationary excess operator builds a new distribution from the one of $X$  by this means:
$$\pr(\mathcal{E}_1(X)> x)=\frac{1}{\er [X]}\int_x^\infty   \pr(X>u) \,du,\quad x \geq 0.$$
The $n^{th}$ iterate $\mathcal{E}_n$, of the operator $\mathcal{E}_1$, builds a sequence of random variables $\mathcal{E}_n(X), \, n\in \mathds{N}$. In \cite{hs1} it was shown that if the sequence $Z_n= \mathcal{E}_n(X) /c_n$ converges in distribution  to some random variable $Z_\infty$ and with some deterministic normalization $c_n$ such that
\begin{equation}\limsup_{n\to \infty} \frac{c_{n+1}}{c_n} \in [1,\infty),
\label{cn}\end{equation}
then, necessarily $\lim_{n\to \infty}c_{n+1}/c_n = l \in [1,\infty)$ and  $\lim_{n\to \infty}\er[Z_n^k] = \er[Z_\infty^k]  \in (0,\infty)$  for every $k\in \mathds{N}$. Many authors were motivated by this problem  and  studied the set of possible distribution for $Z_\infty$, see  the works of  Arratia,  Goldstein and   Kochman \cite{agk},  van Beek and   Braat \cite{bb}, Garcia \cite{gar},   Shantaram  and  Harkness\cite{hs2},  Pakes \cite{pak1,pak2}, Vardi,  Shepp and Logan \cite{vard} for instance. Their approach was mainly based on the fact  that  the distribution of $Z_\infty$  necessarily satisfies
$$\pr(Z_\infty \leq x)=\frac{1}{\er[Z_\infty]}\int_0^{lx} \pr(Z_\infty > x)\, \du, \quad \mbox{for every $x\geq 0$.}$$
The latter is equivalent to the identity in law observed   in Theorem \ref{th3}  below
\begin{equation} Z_\infty  \simdis U_l \,(Z_\infty)_{(1)},
\label{czi}\end{equation}
for $s=1$ and $c=\log l$, and  $U_l$ is uniformly distributed over $[0,l^{-1}]$ and independent from $(Z_\infty)_{(1)}$. We will see that the only significant  information one can extract from identity (\ref{czi})  is that   $Z_\infty$ has the same integer moment as an exponential distribution multiplicatively mixed  by the Log-normal distribution with parameters depending on $l$ and on the value attributed to  $\er[Z_\infty]$. There is then a problem of determinacy in law for $Z_\infty$ since the Log-normal distribution is well known to be moment indeterminate.

Also motivated by this problem of indeterminacy,  we were concerned with the natural question: what  additional information on the distribution of $Z_\infty$ can we obtain if we study the continuous scheme $Z_t= \mathcal{E}_t(X) /c_t$, $t\in [0,\infty)$ where $\mathcal{E}_t$ is the continuous time stationary-excess operator given by
\begin{equation}
\pr(\mathcal{E}_t(X)> x)=\frac{t}{\er [X^t]}\int_x^\infty  (u-x)^{t-1} \pr(X>u) \,du,\quad x \geq 0 \,?
\label{ctse}\end{equation}
Notice that the limit $Z_\infty$ obtained by the discrete scheme has necessarily the same distribution as the one of the continuous scheme.

In section \ref{limel}, we make a digression and entirely formalize  convergence of the Mellin transforms of  general families of random variables (in both discrete and continuous time i.e. $t \in \mathds{N}$ and $t\in [0,\infty)$) through convergence in distribution and uniform integrability. The form (\ref{ctse}), justifies our investigation on $t$-monotone  functions in section  \ref{tmono}.  The results obtained in both sections \ref{limel} and \ref{tmono} will allow us to contribute in section \ref{limit} to the problem raised by  Harkness  and Shantaram \cite{hs1}. We  simplify their problem and solve it as follows:

\begin{enumerate}
\item  condition (\ref{cn}) is not only sufficient, but also necessary for the convergence in distribution of $Z_t$ if we require some integrability on $Z_\infty$;
\item  under (1), convergence of $Z_t$ in  both schemes $t \in \mathds{N}$ or $t \in [0,\infty)$, is equivalent to
$$\frac{X_{(t)}}{\rho_t}\limdis X_\infty \qquad \mbox{as} \; t \to \infty \quad \mbox{and} \quad t\in \mathds{N}\; \mbox{or} \; t \in [0,\infty),$$
the normalization  $\rho_t$ being necessarily equivalent to $t\, c_t$ at infinity;
\item it holds that $Z_\infty \simdis \mathfrak{e} \, X_\infty$ where $\mathfrak{e}$ is independent from $X_\infty $ and is exponentially distributed;
\item the only possible distributions for the limit  $X_\infty$ is  Log-normal.
\end{enumerate}

\bigskip

In what follows, we only deal with nonnegative random variables. The notation $\TT$ stands for  $\mathds{N}$, the set of nonnegative integers, or for the interval  $[0,\infty)$.
\section{Definiteness,  monotonicity and injectivity of the Mellin transform}
The Mellin transform of a nonnegative random variable $X$ is defined by
$$\lambda \mapsto\er[X^\lambda],\quad  \lambda\in \mathcal{D}_X=\{x\in \mathds{R}\,;\,\er[X^x]<\infty \}. $$
We recall H\"older inequality true for every real random variables
\begin{equation}
\er[|UV|]\leq \er[|U|^p]^{\frac{1}{p}}\, \er[|V|^q]^{\frac{1}{p}},\quad p,q>0, \; \frac{1}{p}+\frac{1}{q}=1,
\label{UV}\end{equation}
whenever the expectations are finite. The equality holds in (\ref{UV}) if and only if there exist constants $a , b \geq  0$, not both zero, such that $a|U|^p=b|V|^q$. It is then clear that for a positive random variable $X$, the standard Lyapunov inequality holds:
\begin{equation}
\er[X^{\lambda}]^{\frac{1}{\lambda}}\leq \er[X^{\lao}]^{\frac{1}{\lao}} \quad \mbox{for}\;\; 0<\lambda \leq\lao \quad \mbox{and} \quad
\er[X^\mu]^{\frac{1}{\mu}}\leq \er[X^{\mu_0}]^{\frac{1}{\mu_0}}\; \quad \mbox{for}\;\; \mu_0\leq \mu<0,
\label{ab}\end{equation}
whenever the expectations are finite. The latter justifies that if $\mathcal{D}_X$ contains some  $\lao>0$ (respectively some $\mu_0<0$), then $\mathcal{D}_X$ contains the interval $[0,\lao]$ (respectively  $[\mu_0,0]$). It is then seen that $\mathcal{D}_X$ is an interval with extremities
\begin{equation}
\mu_X =\inf \{\lambda \in \mathds{R}, \; \er[X^\lambda]<\infty\}\quad \mbox{and}\quad\lambda_X =\sup \{\lambda \in \mathds{R}, \; \er[X^\lambda]<\infty\},
\label{rlx} \end{equation}
not necessarily included. Assume $\lambda_X >0$.  By the dominated convergence theorem,  we see that the Mellin transform of $X$ is often differentiable on $(0, \lambda_X)$ and by (\ref{ab}) that
$$\mbox{\it    $\lambda \mapsto \er[X^\lambda]^{1/\lambda}$ is nondecreasing  on $[0,\lambda_X)$.}$$

\noindent The last fact could be also seen as a consequence of this Proposition:
\begin{proposition} For any nonnegative random  variable $X$ such that $\lambda_X>0$, the Mellin transform $\mathcal{M}_X$ is log-convex on $[0,\lambda_X]$.
If furthermore $X$ is non-deterministic, then strict  log-convexity holds.
\label{lcx}\end{proposition}
\begin{proof} Let $g(\lambda) :=\log \er[X^\lambda]$. Trivial computations lead to
$$g''(\lambda) = \frac{\er[X^\lambda \,(\log X)^2 ] \er[X^\lambda]- \er[X^\lambda \,\log  X]^2}{\er[X^\lambda]^2},\quad 0<\lambda <\lambda_X.$$
Taking $p=q=2$, $U=  X^\frac{\lambda}{2}\, \log X$ and $V=X^\frac{\lambda}{2}$ in (\ref{UV}), we deduce that $g$ is convex. It is  strictly  convex unless $U$ and $V$ are proportional which is equivalent to $X$ deterministic.
\end{proof}
Proposition \ref{lcx} gives an additional information:
\begin{corollary} Let a nonnegative random variable $X$ such that  $\lambda_X>0$. For every $\lambda \in (0,\lambda_X)$, the function   $t \mapsto \mathcal{M}_X(\lambda +t)/\mathcal{M}_X(t)$ is nondecreasing on $[0,\lambda_X -\lambda)$. It is further increasing whenever $X$ is non-deterministic.
\label{nef}\end{corollary}
\begin{proof}   Theorem 5.1.1 \cite{webs} p. 194 says that convexity of $x \mapsto  g(x)=\log \mathcal{M}_X(x)$ yields that its slopes are nondecreasing:
$$\frac{g(y)-g(x)}{y-x} \leq\frac{g(z)-g(x)}{z-x}\leq \frac{g(z)-g(y)}{z-y},\quad 0\leq x<y<z<\lambda_X.$$
Then,
\begin{equation}
g(\lambda +s)-g(s)\leq g(\lambda +t)-g(t), \quad \mbox{for $\,0\leq s <t$  and   $\lambda+t < \lambda_X$},
 \label{eq}\end{equation}
and the first assertion is proved. For the strict monotonicity, notice that  equality holds in (\ref{eq}) only in case where the function $r \mapsto g(\lambda +r)-g(r)$ is not injective. Because of the differentiability of $g$,  the latter reads  $g'(\lambda +r)=g'(r)$ for some value of  $r$. The latter is possible only if $g'$ is not injective, that is $g''(x)=0$ for some value  of $x$ and the second statement in Proposition \ref{lcx} allows to conclude.
\end{proof}
Our result in Corollary \ref{nef} is  the same than the one stated in Lemma 3.1 in \cite{hs1}  when  $\lambda$ and $t$ are positive integers. Corollary \ref{nef} is  also proved in  \cite{pak1}, where the author  also adapts the arguments of Lemma 3.1 in \cite{hs1}. However, we found that the argument of continuity used in \cite{pak1},  appealing to a result of Kingman \cite{king},  does not fit his context. We clarify the approach of \cite{pak1}  with this second proof:
\begin{proof}[Second proof of Corollary \ref{nef}] We first show that  the sequence $u_n=\mathcal{M}_X(nu)$ satisfies
\begin{equation}
\frac{u_{n+m}}{u_n} \leq \frac{u_{n+m+1}}{u_{n+1 }},\quad \mbox{for every $n,\,m\in \mathbb{N}, \, u>0$}. \label{cru}
\end{equation}
Schwarz inequality (\ref{UV}), with $p=q=2$, gives  $ \er[X^{d(n+1)}]^2 \leq \er[(X^{u(n+2)}]  \er[X^{un}]$. Then
$$\frac{u_{n+1}}{u_n} \leq \frac{u_{n+2}}{u_{n+1}},\quad \mbox{for every $n\in \mathbb{N}$},$$ which is also equivalent to (\ref{cru}) from which we deduce that  for each $u>0$ and $m \in \mathbb{N}$, the sequence
\begin{equation}
n \mapsto \frac{\er[X^{u(n+m)}]}{\er[X^{un}]} \quad \mbox{is nondecreasing}.
\label{cru1}\end{equation}
Now, take $\lambda>0$ and $t>s>0$ with $s,t$ rationals  of the form $s=p/q$ and $t=k/l$ so that $pl<qk$.  Applying  (\ref{cru1}) with $u=\frac{\lambda}{ql}$, we obtain the inequality
$$\frac{\er[X^{\lambda(s+1)}]}{\er[X^{\lambda s}]}= \frac{\er[X^{\frac{\lambda}{ql}(pl+ql))}]}{\er[X^{\frac{\lambda}{ql} pl}]} \leq \frac{\er[X^{\frac{\lambda}{ql}(qk+ql)}]}{\er[X^{\frac{\lambda}{ql} qk}]} = \frac{\er[X^{\lambda(t+1)}]}{\er[X^{\lambda t}]}.$$
By continuity of the Mellin transform, we deduce that
$$\frac{\er[X^{\lambda(s+1)}]}{\er[X^{\lambda s}]}\leq \frac{\er[X^{\lambda(t+1)}]}{\er[X^{\lambda t}]}, \quad \mbox{for all $\lambda>0$ and all real numbers $s,\,t$ s.t.}\;\,  0<s<t.$$
The proof is finished by replacing the couple $(s,t)$ by  $(\frac{s}{\lambda}, \frac{t}{\lambda})$. Strict monotonicity is obtained as in the end of the first proof.
\end{proof}

We found that in the literature, many papers invoke the injectivity of the Mellin transform without a precise reference. For instance, an informal discussion in exercise  1.13 in \cite{CY} appeals to Chapter VI in Widder's book \cite{widder}  where we found Theorem 6a p. 243  stating the following:
\begin{quote}{\it If the Mellin transforms  of two  nonnegative random variables $X$ and $Y$ are well defined on some strip $\alpha< Re(z) < \beta$, and equal there, then $X\simdis Y$.}
\end{quote}
Widder's theorem could be improved by the following Lemma:
\begin{lemma} Let $X$ and $Y$ two nonnegative random variables such that their Mellin transforms are well defined on some interval $(\alpha, \beta)\subset \mathbb{R}$ and equal there, then $X\simdis Y$.
\label{mellin}\end{lemma}
The proof of this Lemma is a technique borrowed from \cite{aj} and based on a Blaschke's theorem that allows to identify holomorphic functions given their restriction along suitable sequences:
\begin{theorem}[Blaschke,  Corollary  p. 312 in Rudin \cite{rudin}]
If $f$ is holomorphic and bounded on the open unit disc $D$,  if $z_1,z_2,z_3, \cdots$  are the zeros of $f$ in $D$ and if $\sum_{k=1}^\infty  (1 - |z_k|) = \infty,$ then $f(z) = 0$ for all $z \in D$.
\end{theorem}
Using the  one-to-one mapping of the strip $S=\{z \in \mathds{C} , \; 0<Re(z)<1\}$ onto the open unit disc
$$z\mapsto \theta(z) = \frac{e^{\ii \frac{\pi}{2} z}-i}{e^{\ii \frac{\pi}{2} z}+i},$$
one can easily rephrase  Blaschke's theorem for function defined on the strip $S$:
\begin{corollary}\label{blas}
Two holomorphic functions on the strip $S$ are identical if their difference is bounded and if they coincide along a sequence  $\alpha_1,\alpha_2,\alpha_3, \cdots$ in $S$,
such that the series $\sum_k   \left(1 - \left|\frac{e^{\ii \frac{\pi}{2} \alpha_k}-i}{e^{\ii \frac{\pi}{2} \alpha_k}+i}\right|\right)$ diverge. For instance, the series diverge for the sequence  $\alpha_k=\frac{1}{k}, \; k\geq 1$.
\end{corollary}
\begin{proof}[Proof of Lemma \ref{mellin}] It is enough to take  $(\alpha, \beta)=(0,1)$, to notice that both $\mathcal{M}_X$ and  $\mathcal{M}_Y$ extend holomorphically on the strip $S$  and to conclude with Corollary \ref{blas} since $\mathcal{M}_X$ and  $\mathcal{M}_Y$ coincide along the sequence $\alpha_k=k^{-1}, \; k\geq 2$ which is contained in $(0,1)$.
\end{proof}
\section{Mellin transform and  size biased laws}
Let  $X$ denote a non-deterministic nonnegative random variable.  For $t\in \mathcal{D}_X$, the size biased law of
order $t$ is denoted  $X_{(t)}$ and is a version of the  weighted law
\begin{equation}\label{weig}
\pr (X_{(t)} \in dx) =  \frac{x^t}{\er[X^t]} \pr (X  \in dx),\quad x\geq 0.
\end{equation}
Chebychev's association inequality  says that
$$\er[f(X)g(X)] \geq \er[f(X)]\,\er[g(X)], $$
whenever the expectations are well defined and $f,\,g$ are both
nondecreasing or nonincreasing real-valued functions.  Taking $f(u)=u^t$, $g(u)={\II}_{u>x}$, we see that $X_{(t)}  \geq_{st}   X$ for $t\geq 0$,  i.e.
\begin{equation}\pr(X_{(t)}  > x) = \frac{\er [X^t \,{\II}_{X>x}]}{\er[X^t]}\geq  \pr(X > x),\quad \forall t,\,x  \geq 0\,.
\label{asso}\end{equation}
Notice that the last stochastic inequality also justifies Corollary \ref{nef}, since the Mellin transform could be computed as
\begin{equation}
\er[X^\lambda]=\lambda \int_{_0}^\infty x^{\lambda -1} \, \pr(X > x)\, \dx,\quad \lambda \in (0,\lambda_X) ,
\label{rep}\end{equation}
whenever  the extremity $\lambda_X$ given by (\ref{rlx}) is positive.

We list some elementary properties for the size biased law of a r.v. $X$:

\noindent (P0)  For every $c>0$ and $t \in \mathcal{D}_X$, we have
$$X _{(0)}=  X  \quad \mbox{and}  \quad  (c X )_{(t)}= c X_{(t)}$$
\noindent (P1)  For every $\lambda, t$ such that
$t, \,t+\lambda \in \mathcal{D}_X$ and measurable bounded function $g$, we have
$$ \er[X_{(t)}^\lambda] = \frac{\er[X^{t+\lambda}]}{\er[X^t]}, \qquad \er[g(X_{(t)})] = \frac{\er[ X^t g(X)]}{\er[X^t]}.$$
\noindent (P2) For every $s,t \in \mathcal{D}_X$ such that $t+s \in \mathcal{D}_X$, we have
$$\left(X_{(s)}\right)_{(t)}\simdis X_{(s+t)}\simdis \left(X_{(t)}\right)_{(s)}.$$
\noindent(P3) For every $s,t$ such that $st \in \mathcal{D}_X$, we have
$$(X^s)_{(t)} \simdis  ( X_{(st)})^s.$$
\noindent (P4) For every  independent random variables  $X$, $Y$ and  $t \in \mathcal{D}_X \cap \mathcal{D}_Y$, we have
$$(X Y)_{(t)}\simdis X_{(t)}  Y_{(t)}\qquad (\mbox{assuming that $X_{(t)}$ and $Y_{(t)}$ are independent}).$$

\section{Convergence of sequences and families of Mellin transforms}\label{limel}

This section contains results  dealing with sequences or families of Mellin transforms. As we did for the injectivity in Lemma \ref{mellin}, we felt  it was important to also clarify the notion of convergence via Mellin transform. Next Theorem \ref{th2} and Proposition \ref{melib}, will be crucial for handling section \ref{limit}  below.\\

We recall that $\TT=\mathds{N}$ or $[0,\infty)$.  In what follows,
\begin{enumerate}
\item  a property $(\mathcal{P}_t)$ is said to be {\it true for $t$ big enough}, if there exists $t_{_0}\in \TT$ such that $(\mathcal{P}_t)$ is true for $t\geq t_{_0}$;

\item  ${(X_t)}_{t\in \TT}$ denotes a family of nonnegative random
variables indexed by the time $t\in \TT$;

\item  by a subsequence of ${(X_t)}_{t\in \TT}$ , we mean a collection of  random variables ${(X_{t(n)})}_{n\in \mathbb{N}}$ obtained through
a nondecreasing function $t:\mathbb{N}\to \TT$ such that $t(n)\to
\infty$ as $n\to \infty$;

\item  we always assume that for $t$ big enough,
$$\lambda_{X_t}=\sup \{\lambda  \in \mathds{R}, \; \er[X_t^\lambda]<\infty\} >0;$$

\item   for $\lambda\geq 0$, we define informally
\begin{equation}
m(\lambda):=\liminf_{t\in \TT} \er[X_t^{\lambda}] \quad \mbox{and}\quad M(\lambda):=\limsup_{t\in \TT} \er[X_t^{\lambda}].
\label{mM}\end{equation}
\end{enumerate}

We also recall some basic ingredients related to convergence in distribution.
 \begin{definition}[Billingsley \cite{Bil}]  Let a sequence ${(X_n)}_{n\in \mathds{N}}$ of real-valued random
variables.

 (i)  ${(X_n)}_{n\in \mathds{N}}$ is called  tight  if
 $$\sup_{n\in \mathds{N}}\mathbb{P}(|X_n|>x) \to 0 \quad \mbox{as}\quad x \to \infty.$$

 (ii)  ${(X_n)}_{n\in \mathds{N}}$ is called  uniformly integrable  if
$$\sup_{n\in \mathds{N}} \er[|X_n| {\II}_{X_n>x}] \to 0 \quad \mbox{as}\quad x \to \infty.$$

 (iii) $X_n$ converges in distribution to $X_\infty$ if $\er[f(X_n)] \to \er[f(X_\infty)]$, as $n\to \infty$, for every continuous bounded
(or continuous  compactly supported) real function $f$.
\label{bilc}\end{definition}
We are going to study convergence in distribution for families of nonnegative random variables. For this purpose, we   slightly generalize Definition \ref{bilc} in order to get more flexibility.
 \begin{definition} Let family ${(X_t)}_{t\in \TT}$ a family of nonnegative random variables.

  (i) We say that the family  ${(X_t)}_{t\in \TT}$    ultimately tight  if
 $$\limsup_{t\in \TT}\mathbb{P}(X_t >x) \to 0, \quad \mbox{as}\;\, x \to \infty.$$

(ii)  We say that the family ${(X_t)}_{t\in \TT}$  is  $\lambda$-uniformly integrable, and we denote  ${(X_t)}_{t\in \TT}$ is $\lambda-UI$,   if  $\lambda \in (0, \lambda_{X_t})$ for $t$ big enough  and
$$\limsup_{t\in \TT} \er[X_t^{\lambda} {\II}_{X_t>x}] \to 0, \quad \mbox{as}\;\, x \to \infty.$$

(iii) We say that the family ${(X_t)}_{t\in [0,\infty)}$ converge  in distribution, if every subsequence ${(X_{t(n)})}_{n\in \mathbb{N}}$ converges in distribution in the usual sense (iii) of the preceding Definition.
\label{def2}\end{definition}
 \begin{remark}{\rm We can notice that:

a) $\lambda-$uniform  integrability  of ${(X_t)}_{t\geq 0}$ is equivalent to $1-$uniform  integrability of ${(X_t^\lambda)}_{t\geq 0}$.

b) If $\,\TT=\mathds{N}$, then  Definitions  \ref{bilc} and \ref{def2} are equivalent  $0\lambda_{X_t} \in (0,\infty)$ for every $t \in \mathds{N}$. In general,  this is untrue if $\,\TT=[0,\infty)$.

c) If $\,\TT=[0,\infty)$, then  ${(X_t)}_{t\geq 0}$ is ultimately tight (respectively $ \lambda-UI$) if and only if there exists  some positive $t_{_0}$, big enough, such that ${(X_t)}_{t\geq t_{_0}}$ is tight (respectively ${(X_t^\lambda)}_{t\geq t_{_0}}$ is uniformly  integrable) in the same sense that Billingsley gave in Definition \ref{bilc}.
}\label{RUI} \end{remark}
We start this section with the following result that clarifies the link between ultimate tightness and uniform integrability:
\begin{proposition} Let $\lao>0$ and  ${(X_t)}_{t\in \TT}$ a family of nonnegative random variables. Recall the function $m(.)$ and $M(.)$ are given by (\ref{mM}).

\noindent 1) If ${(X_t)}_{t\in \TT}$ is  $\lao-$uniformly integrable, then $M(\lao)<\infty$.

\noindent 2) If $M(\lao)<\infty$, then ${(X_t)}_{t\in \TT}$ is ultimately tight and also $\lambda-$uniformly integrable for every $\lambda  \in(0,\lao)$.

\noindent 3) Assume $m(\lao)>0$ and $M(\lao+\epsilon)<\infty, \;$  for some $\,\epsilon>0$. Then $\lao$-uniform integrability of
${(X_t)}_{t\in \TT}$  is equivalent to its ultimate tightness.
\label{prop0}\end{proposition}
\begin{proof} 1)   Write
$$M(\lao)\leq x^\lambda + \limsup_{t\in \TT}  \er[X_t^\lambda {\II}_{X_t>x}],$$
for $x$ big enough and deduce that $M(\lao)<\infty$.

\noindent 2) H\"older and Markov inequalities give
\begin{equation}
\limsup_{t\in \TT} \er[X_t^\lambda {\II}_{X_t>x}]\leq M(\lao)^{\frac{\lambda}{\lao}} \limsup_{t\in \TT} \pr(X_t>x)^{\frac{\lao-\lambda}{\lao}}\leq
\frac{M(\lao)}{x^{\lao-\lambda}},\quad 0<\lambda<\lao,
\label{thg}\end{equation}
Ultimate  tightness  and $\lambda$-uniform integrability are then immediate.

\noindent 3) Inequality (\ref{asso}) gives
$$ \er[X_t^{\lao}]  \pr(X_t>x) \leq \er[X_t^{\lao} {\II}_{X_t>x}],\qquad \mbox{for all $x>0$ and $t\in \TT$}.$$
Using again inequality (\ref{thg}), with the couple $(\lambda,\lao)$ replaced by $(\lao,\lao+\epsilon)$,  obtain
$$ m(\lao) \limsup_{t\in \TT} \pr(X_t>x) \leq \limsup_{t\in \TT}\er[X_t^{\lao} {\II}_{X_t>x}]\leq M(\lao+\epsilon)^{\frac{\lao}{\lao+\epsilon}}  \limsup_{t\in \TT} \pr(X_t>x)^{\frac{\epsilon}{\lao+\epsilon}}.$$
\end{proof}
\noindent Next Theorem rephrases and improves some results borrowed from the monograph of Billingsley \cite{Bil}:
\begin{theorem} Let   ${(X_t)}_{t\in \TT}$ a family of nonnegative random variables such that   $\lao \in (0, \lambda_{X_t})$, for some  $\lao >0$ and   $t$ big enough.

\noindent 1) Let $X_\infty$ a nonnegative random variable. The following assertions are equivalent, as $t\to \infty$:

(i)  $X_t \limdis X_\infty\,$  and ${(X_t)}_{t\in \TT}$ is $\lao$-uniformly integrable;

(ii)  $X_t \limdis X_\infty\,$, $\lao \in \mathcal{D}_{X_\infty}$ and $\er[X_t^\lao] \to  \er[X_\infty^\lao]\,$;

(iii) $\lao \in \mathcal{D}_{X_\infty}$  and  for every $\lambda \in [0, \lao]$,  $\er[X_t^\lambda] \to  \er[X_\infty^\lambda]$.

\noindent 2) Let  $ \lambda_1 \in(0, \lao)$ and  assume that   $\,\er[X_t^\lambda]$ converges as $t\to \infty$ to a well defined function $f(\lambda), \; \lambda \in [\lambda_1 ,\lao]$. Then (iii) holds  and $f$ is well defined on $[0,\lao]$ by $f(\lambda)=\er[X_\infty^\lambda]$.
\label{th2}\end{theorem}
\begin{proof} The proof is conducted by reasoning on subsequences.

\noindent 1) \underline{$(i)\Longrightarrow (iii)$}: it is a direct application of Theorem 25.12 p. 338 in \cite{Bil},  using Remark \ref{RUI}  and  Proposition \ref{prop0}.

\noindent  \underline{$(iii)\Longrightarrow (ii)$} is treated as follows: Since $M(\lao)<\infty$, then by Proposition \ref{prop0}, the family  ${(X_t)}_{t\in \TT}$ is ultimately tight. Lemma \ref{mellin} insures that   any subsequence ${(X_{t(n)})}_{n\in \mathbb{N}}$, if  converging in distribution  as $n \to \infty$, necessarily converge to the law of $X_\infty$.  Corollary in \cite{Bil} p.337  allows to conclude that $X_t \limdis X_\infty$ as $t \to \infty$.

\underline{$(ii)\Longrightarrow (i)$}: we use the following representation valid for any nonnegative random variables $Z$ such that $\er[Z^\lambda]<\infty$:
$$\er[Z^\lambda\II_{Z\leq x}]=x^\lambda \pr(Z\leq x)-\lambda \int_{_0}^x u^{\lambda -1}\, \pr(Z\leq u)\, du, \quad x\geq 0.$$
Choose $x_{_0}$  a continuity point of $u\mapsto \pr(X_\infty\leq u)$. By the dominated convergence theorem, we have
\begin{eqnarray*}
\lim_{t\to \infty}\er[X_t^\lambda\II_{X_t\leq x_{_0}}]&=&\lim_{t\to \infty}\left(x^\lambda \pr(X_t\leq x)-\lambda \int_{_0}^{x_{_0}} u^{\lambda -1}\, \pr(X_t\leq u) \, du\right)\\
&=& x^\lambda \pr(X_\infty\leq x_{_0})-\lambda \int_{_0}^{x_{_0}} u^{\lambda -1}\, \pr(X_\infty\leq u) \, du= \er[X_\infty^\lambda\II_{X_\infty\leq x_{_0}}]
\end{eqnarray*}
 Since $\lim_{t\to \infty}\er[X_t]=\er[X_\infty]$, we also have $\lim_{t\to \infty}\er[X_t^\lambda\II_{X_t > x_{_0}}]=\er[X_\infty^\lambda\II_{X_\infty > x_{_0}}]$,
 that is, for every $\epsilon >0$, there exists $t_{_0} \in\TT$ such   $\left|\er[X_t^\lambda\II_{X_t > x_{_0}}]-\er[X_\infty^\lambda\II_{X_\infty> x_{_0}}]\right|<\epsilon$.  Now choose   $\epsilon >0$, then $x_{_0}$ big enough so that $\er[X_\infty^\lambda\II_{X_\infty > x_{_0}}]<\epsilon$. We deduce that
$$\er[X_t^\lambda\II_{X_t > x}] \leq \er[X_t^\lambda\II_{X_t > x_{_0}}]<\er[X_\infty^\lambda\II_{X_\infty> x_{_0}}] +\epsilon <2\epsilon, \quad t\geq t_{_0},\;x\geq x_{_0}.$$

\noindent 2) We adapt a part of the proof of Corollary 1.6 p. 5 given in Schilling and al. \cite{SSV} in the context of convergence of sequences on completely monotone functions. Helly's selection theorem allows a shortcut since there exists a subsequence
 ${(X_{t(n)})}_{n\in \mathbb{N}}$  satisfying  $X_{t(n)} \limdis X_\infty$, as $n \to \infty$.

Fix  $\lambda \in [\lambda_1, \lao]$. For every function $h:[0,\infty) \to [0,1]$, compactly supported, we find, by Fatou Lemma, that
$$\er[h (X_\infty) \; X_\infty^\lambda]=\lim_{s\to \infty} \er[h (X_{t(n)}) \; X_{t(n)}^\lambda]\leq \lim_{s\to \infty} \er[ X_{t(n)}^\lambda]= f(\lambda).$$
Monotone convergence theorem gives a first inequality
$$\er[X_{\infty}^\lambda]=sup_{h}\er[h (X_\infty) \; X_{\infty}^\lambda]  \leq f(\lambda).$$
Now, fix $\epsilon >0$,  choose a continuity point  $x$ of the
distribution function of $ X_\infty$, then apply the fact that
$X_{t(n)} \limdis X_\infty$, identity (\ref{rep}) and the dominated
convergence theorem, in order to get that for $n$ big enough,
$$\er[ X_{t(n)}^\lambda\,{\II}_{ X_{t(n)}\leq x}] - \er[ X_\infty^\lambda\,{\II}_{ X_\infty\leq x}]=  \lambda\int_{_0}^x u^{\lambda -1} \left(\mathbb{P}(u<X_{t(n)}\leq x) - \mathbb{P}(u<X_\infty\leq x)\right)  du  \leq  x^\lambda\, \epsilon.$$
Since ${(X_{t(n)})}_{n\in \mathbb{N}}$ is $\lambda-UI$, then $\er[ X_{t(n)}^\lambda\,{\II}_{ X_{t(n)}> x}]<\epsilon$  for all $x,\,n$ big enough. Finally get for all $\epsilon>0$,
$$\er[ X_{t(n)}^\lambda] - \er[ X_\infty^\lambda]\leq  \er[ X_{t(n)}^\lambda\,{\II}_{ X_{t(n)}> x}]   + \er[ X_{t(n)}^\lambda\,{\II}_{ X_{t(n)}\leq x}] - \er[ X_\infty^\lambda\,{\II}_{ X_\infty\leq x}]< \epsilon(1+ x^\lambda).$$
The latter proves the second inequality $f(\lambda)=\lim_{s\to \infty} \er[ X_{t(n)}^\lambda]\leq \er[ X_\infty^\lambda]$. All in one, we have that
$$f(\lambda)= \er[ X_\infty^\lambda], \qquad \mbox{for every}\;\lambda \in [\lambda_1, \lao].$$
As in point 1) above, notice that  the family  ${(X_{t})}_{t\in \TT}$ is  ultimately tight, and  by Lemma \ref{mellin}, each subsequence of it, if converging in distribution, necessarily converge to the distribution of $X_\infty$. Use again the  Corollary in \cite{Bil} p. 337 in order to have $X_t \limdis X_\infty$, as $t\to \infty$. To conclude, use $\lambda$-uniform integrability of ${(X_{t})}_{t\in \TT}$  and then implication $(i) \Longrightarrow (ii)$ in point 1) above.
\end{proof}

Now, consider two families of nonnegative random variables ${(U_t)}_{t\in \TT}$ and ${(V_t)}_{t\in \TT}$  such that  $U_t\limdis U_\infty$,  $\,V_t\limdis V_\infty$ and  $U_t$ and $V_t$  independent for every $t\in \TT$. It then is trivial that  $U_t V_t\limdis U_\infty V_\infty$.   As a consequence of Theorem \ref{th2}, we deduce a kind of converse:
\begin{corollary} Let  ${(U_t)}_{t\in \TT}$ , ${(V_t)}_{t\in \TT}$ and ${(W_t)}_{t\in \TT}$ three families of nonnegative random variables such that
$U_t$ and $V_t$ are independent  for  each $t\in \TT$ and such that

\noindent (i)  the factorizations in law $ W_t \simdis U_t \,V_t $  holds;

\noindent (ii)  the  convergences in distribution  $W_t\limdis U_\infty\;$ and $\; V_t\limdis V_\infty \neq 0$ hold as $t\to \infty$;

\noindent (iii) there exists $\lao >0$ such that ${(W_t)}_{t\in \TT}$ is $\lao-UI$ or such that $\lim_{t\in \TT}\er[W_t^{\lao}]<\infty$.

Then, $U_t\limdis U_\infty$, where the distribution of the  random variable $U_\infty$ is well defined by its Mellin transform given by    $\er[U_\infty^\lambda]=\er[W_\infty^\lambda]/\er[V_\infty^\lambda]$, for every $\lambda  \in [0,\lao]$.
\label{XYZ}\end{corollary}
\begin{proof} Notice that $\lao \in \mathcal{D}_{W_t} = \mathcal{D}_{U_t} \cap \mathcal{D}_{V_t}$ for $t$ big enough and that both conditions in $(iii)$ are equivalent by Theorem \ref{th2}. Choose $v_{_0}>0$ a continuity point of $x\mapsto \pr(V_\infty > x)$ such that $\pr(V_\infty > v_{_0})>1/2$ and notice also that for  every $t\in \TT$ and $x\geq 0$,
\begin{eqnarray*}\er[W_t^{\lao} {\II}_{W_t>x}]&\geq&\er[(U_t V_t)^{\lao} {\II}_{U_t>\frac{x}{v_{_0}},\,V_t>v_{_0} }]=\er[U_t^{\lao} {\II}_{U_t>\frac{x}{v_{_0}}}]\, \er[V_t^{\lao} {\II}_{V_t>v_{_0}}]\\
&\geq& \er[U_t^{\lao} {\II}_{U_t>\frac{x}{v_{_0}}}]\;
v_{_0}^{\lao} \;\pr(V_t>v_{_0}).
\end{eqnarray*}
Then use the fact that there exists $t_{_0} \in \TT$ such that $\pr(V_t>v_{_0})> \pr(V_\infty>v_{_0})-1/4>1/4$ for $t\geq t_{_0}$ and then
$$\er[W_t^{\lao} {\II}_{W_t>x}]\geq \frac{v_{_0}}{4}\, \er[U_t^{\lao} {\II}_{U_t>\frac{x}{v_{_0}}}], \quad t\geq t_{_0},\,x\geq 0.$$
The latter yields that the family  ${(U_t)}_{t\in \TT}$  is $\lao -UI$, then  apply Theorem \ref{th2}.
\end{proof}
Next proposition studies the convergence of biased laws and  improves  Theorem 2.3 in \cite{agk}:
\begin{proposition} Let  ${(X_t)}_{t\in \TT}$ a family of nonnegative random variables  such that  $X_t$ converges in distribution to a non-null random $X_\infty$. Suppose  $0<\lao < \min(\lambda_{X_t}\,,\, \lambda_{X_\infty})$ for $t$ big enough and  $\lim_{t \to \infty}\er[X_t^{\lao}] \to \er[X_\infty^{\lao}].$ Then, we have the convergence of the   size biased distributions  of ${(X_t)}_{t\in \TT}$:
\begin{equation}
(X_t)_{(\lambda)} \limdis (X_\infty)_{(\lambda)}, \quad \mbox{as $t \to \infty$}, \qquad \mbox{for every $\lambda\in [0,\lao]$}.
\label{bi}\end{equation}
\label{melib}\end{proposition}
\begin{proof} a) We start by proving (\ref{bi}) for $\lambda=\lao$. By assumption, we have $0<\er[X_\infty^{\lao}]<\infty$. Convergence in distribution of $X_t$ to $X_\infty$ is equivalent to $\er[g(X_t)] \to \er[g(X_\infty)]$ for every continuous, compactly supported function $g$, as $t \to \infty$. The function $h(x)=|x|^{\lao} g(x)$ is also a continuous, compactly supported function. By property (P1), we also have
$$\er[h(X_t)]= \er[X_t^{\lao}] \er[g\left((X_t)_{({\lao})}\right)] \to \er[h(X_\infty)]=  \er[X_\infty^{\lao}] \er[g\left((X_\infty)_{({\lao})}\right)].$$
The limit (\ref{bi}) for $\lambda=\lao$ follows by simplification in both sides of the last limit.

b) By Proposition \ref{prop0}, notice that ${(X_t)}_{t\in \TT}$  is $\lambda-UI$ for every $\lambda\in (0,\lao)$. Deduce by Theorem \ref{th2} that  $\lim_{t \to \infty}\er[X_t^\lambda] \to \er[X_\infty^\lambda]$ for and reproduce step a) for $\lambda\in [0,\lao)$.
\end{proof}
\begin{remark}{\rm  By Theorem \ref{th2}, finiteness of the quantity $M(\lao)$ given by (\ref{mM}), or $\lao$-uniform integrability of ${(X_t)}_{t\in \TT}$ is sufficient to insure that $\lao \in  \mathcal{D}_{X_\infty}$.
}\end{remark}
\section{$\mathbf{t}$-monotone density functions}\label{tmono}
Let $x_+$ denotes $\max\{0,x\},\, x\in \mathds{R}$. The following definition extends the one of  Schilling et al.  \cite{SSV} given for $t$ nonnegative integer.
\begin{definition} Let $t \in (0, \infty)$. A function $f : (0,\infty) \to [0,\infty)$  is $t$-monotone if it is  represented by
\begin{equation}
f(x)= c+\int_{(0,\infty)} (u-x)_+^{t-1} \, \nu(du),\quad x>0
\label{tmo}\end{equation}
for some $c \geq 0$ and some measure $\nu$ on $(0,\infty)$.
\end{definition}
\begin{remark}{\rm  When $t=1$, representation (\ref{tmo}) holds if and only if $f$ is nonincreasing and right-continuous. When $t=n$ is an integer greater than or equal to 2, representation (\ref{tmo}) holds if and only if $f$  is $n-2$ times differentiable, $(-1)^j f^{(j)}(x) \geq 0$  for all $j= 0, 1, \cdots, n-2$ and $x>0$, and $(-1)^{n-2} f^{(n-2)}$ is nonincreasing and convex. Furthermore, by Theorem 1.11 p.8, \cite{SSV}), the couple $(c,\nu)$ in (\ref{tmo}) uniquely determines $f$.
}\end{remark}

A random variable    $\mathfrak{b}_{a,b}$ is said to have the beta distribution with parameter $(a,b), \, a,b>0,$ if it has the density function
$$\frac{1}{\beta(a,b)}\, x^{a-1}\,(1-x)^{b-1}, \, x\in (0,1)\quad \mbox{with}\quad   \beta(a,b)=\frac{\Gamma(a)\Gamma(b)}{\Gamma(a+b)}.$$
A random variable    $\mathfrak{g}_a$ is said to have the Gamma distribution with parameter $a>0,$ if it has the density function
${\Gamma(a)}^{-1}\,{x^{a-1}}\,e^{-x},\;x\in (0,\infty).$
It is well known that
\begin{equation}
\mathfrak{g}_a \simdis \mathfrak{b}_{a,b} \,\mathfrak{g}_{a+b} \quad
\mbox{and}\quad \mathfrak{b}_{a,b+c}\simdis \mathfrak{b}_{a,b}\,
\mathfrak{b}_{a+b,c}, \quad \mbox{for all} \;  a,b,c>0,   \,,
\label{algb}\end{equation} where in the first (respectively second)
identity $\mathfrak{b}_{a,b}$ and $\mathfrak{g}_{a+b}$ (respectively
$\mathfrak{b}_{a,b}$ and $\mathfrak{b}_{a+b,c}$)  are assumed to be
independent. Biasing on Beta and Gamma variables is nicely expressed by
\begin{equation}(\mathfrak{b}_{a,b})_{(t)}\simdis \mathfrak{b}_{a+t,b},  \quad\mbox{and} \quad (\mathfrak{g}_{a})_{(t)}\simdis \mathfrak{g}_{a+t}, \quad \mbox{for all} \; t>0.
\label{bia}\end{equation}
In the sequel, we denote $\mathfrak{b}_t$ the random
variable defined by $\mathfrak{b}_0=1$ and $\mathfrak{b}_{1,t}$ if $t>0$, that is
$$\mathfrak{b}_t \; \mbox{has the density function}\; t (1-x)^{t-1}, \, x\in (0,1), t>0.$$
Also, $\mathfrak{e} = \mathfrak{g}_1$ denotes a random variable with standard exponential  distribution.
It is clear that
\begin{equation}
\mathfrak{b}_t \simdis 1- e^{- \textstyle \frac{\mathfrak{e}}{t}}\qquad
\mbox{and that}\qquad  t \mathfrak{b}_t \,\limdis \, \mathfrak{e}, \quad
\mbox{as}\; t \to \infty. \label{bb} \end{equation}

We propose the following characterization for  $\mathbf{t}$-monotone densities.
\begin{proposition} Let $t \in (0, \infty)$.

\noindent  1) The density function  $f : (0,\infty) \to [0,\infty)$  of a positive random variable $Z$, is $t$-monotone, if and only if there exists a positive random variable $Y_t$ such that  $f$ is  represented
\begin{equation}
f(x)= t \int_{(0,\infty)} \left(1-\frac{x}{u}\right)_+^{t-1} \,  \frac{\mathbb{P}(Y_t \in du)}{u},\quad x>0.
\label{reph}\end{equation}
This equivalent to the factorization in law $\,Z \simdis \mathfrak{b}_t \, Y_t,  \,$ where $\mathfrak{b}_t$ has the beta distribution as in (\ref{bb}) and is independent from $Y_t$.

\noindent  2) If $f$ is $t$-monotone, then it is also $s$-monotone for every $s \in (0, t)$.

\noindent  3) Furthermore, the $\nu$-measure associated to $f$ through (\ref{tmo}) is finite if and only if there exists a positive random variable $X$ such that $Y_t $ has the same distribution as the size biased random variable $X_{(t)}$ given by (\ref{weig}) i.e.
\begin{equation}
Z \simdis \mathfrak{b_t} \,X_{(t)}.
\label{beta}\end{equation}
\end{proposition}
\begin{remark} {\rm Bernstein characterization for completely monotone functions says that a  function $f:(0,\infty)\to \mathds{R}$    is $n$-monotone for every $n\in \mathds{N}$, if and only if it is represented as the Laplace transform of a (unique) Radon measure $\nu$ on $[0,\infty)$:
\begin{equation}\label{lap}
f(x)=\int_{[0,\infty)} e^{-x u} \nu (\du), \quad \lambda >0\;\; \mbox{(respectively    $\; x \geq0$)};
\end{equation}
When $f$ is a density function associated to a positive random variable $Z$, the latter is equivalent to $Z\simdis \,\mathfrak{e} Y$ where $Y$ is positive and independent from the exponentially distributed random variable $\mathfrak{e}$ and also to $f(x)=\er [e^{-x/Y}/Y], \, x>0$. In this case, Bernstein characterization for $f$ could be reinterpreted as follows:

\noindent - use  the Beta-Gamma algebra (\ref{algb}) in order to write
\begin{equation}
Z\simdis \mathfrak{e}  \simdis \mathfrak{b}_t  \, \mathfrak{g}_t Y, \quad \pr(Z>x)=\er \left[\left(1-\frac{x}{ \mathfrak{g}_t Y}\right)_+^t\right], \quad \mbox{for every}\; t \,x>0;
\label{zey} \end{equation}
\noindent -  use the fact that   $(1-\frac{x}{t})_+^t \to e^{-x}$ uniformly in $x>0$, as $t\to \infty$, and that $\frac{\mathfrak{g}_t}{t} \limdis 1$, then rephrase (\ref{zey})   as:
$$\pr(Z>x)= \lim_{n\to \infty}\er \left[\left(1- \frac{t}{\mathfrak{g}_t}\,\frac{x}{t Y }\right)_+^n\right]= \er[e^{-x/Y}]=\pr(\mathfrak{e}\,Y>x), \quad \mbox{for every} \,x>0. $$
This clarifies the discussion made  right after Proposition 2.2 in \cite{ll}.
}\end{remark}
\begin{proof}[Proof of Proposition \ref{beta}] 1) The density function  $f$ is of the form (\ref{tmo}) if and only if $c=0$ and
$$f(x)=\int_{(0,\infty)} (1-\frac{x}{u})_+^{t-1} \, u^{t-1}\nu(du),\quad x>0,$$
some measure $\nu$ on $(0,\infty)$ such that  $1=\int_{_0}^\infty f(x) dx = \int_{(0,\infty)} t^{-1} u^{t} \,  \nu(du)$ so that $t^{-1} u^{t} \, \nu(dx)$ is a probability measure associated to some random variable, say $Y_t$. The second assertion  is due to the fact that the density of the independent product of a  non negative random variables $U$ and $V$ such that $U$ has a density function $f_U$ is given by the Mellin convolution
$$f_{UV}(x)= \int_{(0,\infty)} f_U(\frac{x}{y})  \,  \frac{\mathbb{P}(V \in dy)}{y},\quad x>0.$$
2) It is enough to use the Beta-algebra (\ref{algb}):  $\mathfrak{b}_t\simdis \mathfrak{b}_s \mathfrak{b}_{1+s,t-s}$.

\noindent 3)  $\er[Y_t^{-t}]= t^{-1} \nu(0,\infty) <\infty$  is equivalent to the identity $Y_t\simdis X_{(t)}$ where $X$ has the distribution $\nu(0,\infty)^{-1} \nu(x)$.
\end{proof}
\section{Application: Refinement of the results of Harkness and Shantaram \cite{hs1}} \label{limit}

Recall that $\TT=\mathds{N}$ or $[0,\infty)$. In what follows $X$  is a nonnegative random variable, non identically null, such that $[0,\infty)\subset \mathcal{D}_X$ i.e. $X$ has moments of all positive orders.

We are willing to obtain a  limit theorem for the family obtained by size biasing the distribution of $X$. Identity (\ref{beta}) suggests to introduce a family of random variables ${(\mathcal{E}_t(X))}_{t \in \TT}$, such that each $\mathcal{E}_t(X)$ has its distribution defined through an operator $\mathcal{E}_t$ stemming from the identity
\begin{equation}
\mathcal{E}_t(X):\simdis \mathfrak{B_t} \,X_{(t)}, \quad \mbox{with with $\mathfrak{B_t}$ independent from $X_{(t)}$.}
\label{beta1}\end{equation}
By property (P2) and identity (\ref{algb}) for Beta distributions, notice that the family ${(\mathcal{E}_t)}_{t \in \TT}$ forms a semigroup of commuting operators:
$$\mathcal{E}_t (\mathcal{E}_s(X))\simdis \mathcal{E}_t (\mathcal{E}_s(X))\simdis \mathcal{E}_{t+s}(X),\quad s,\,t\in \TT.$$
By simple computations, we obtain that identity (\ref{beta1}) is equivalent to one of the following expressions for the Mellin transform or for the distribution function of $\mathcal{E}_t(X)$:
$$\frac{\er[\mathcal{E}_t(X)^\lambda]}{\Gamma(\lambda+1)}= \frac{\Gamma(t+1)}{\er [X^t]}  \frac{\er [X^{\lambda+t}]}{\Gamma(\lambda+t+1)}=
\frac{\Gamma(t+1)}{\Gamma(\lambda+t+1)}   \er \left[X_{(t)}^{\lambda}\right] ,\quad \lambda \geq 0,$$
$$\pr(\mathcal{E}_t(X)>x)=\er \left[(1-\frac{x}{X_{(t)}})_+^t\right]= \frac{\er [(X-x)_+^t]}{\er [X^{t}]}=\frac{t}{[X^{t}]}\int_{_x}^\infty (u-x)^{t-1} \pr(X>u) \,du,\quad x \geq 0,$$
so that
$$\pr(\mathcal{E}_1(X)\leq x)=\frac{1}{\er [X]}\int_{_0}^x \pr(X>u) \,du,\quad x \geq 0.$$
It is then clear that the operator $\mathcal{E}_1$ corresponds to the stationary excess operator studied by  Harkness and Shantaram \cite{hs1} and also in \cite{hs2,pak1,pak2,vard}. It is also seen that the operator $\mathcal{E}_n$ corresponds the $n$-th iterate by  the composition of $\mathcal{E}_1$:
$$ \mathcal{E}_{n+1}=\mathcal{E}_1 \circ  \mathcal{E}_{n},\quad n\in \mathds{N}\rsetminus \{0\}.$$

Harkness and Shantaram \cite{hs1} solved the discrete time problem ($\TT=\mathds{N}$) of:

- finding a  deterministic normalization  speed $c_t, \; t \in \mathbb{N}$, and sufficient conditions such that
\begin{equation}
Z_t:=\frac{\mathcal{E}_t(X)}{c_t}   \; \limdis \;Z_\infty\quad \mbox{as}\; t\in \mathds{N} \;  \mbox{and}\;t\to \infty.
\label{pb0}\end{equation}

- describing the set of possible  distributions for $Z_\infty$.

\noindent  It is natural to study what kind of additional information we can recover from the continuous time problem,  i.e. convergence (\ref{pb0}) in case $\TT=[0,\infty)$ instead of $T=\mathds{N}$, and to find the necessary and  sufficient conditions such that
\begin{equation}
Z_t \simdis \frac{1}{c_t} \mathfrak{b}_t \,X_{(t)} \; \limdis \;Z_\infty\quad \mbox{when $t\in \TT$ and}\; t \to \infty.
\label{pb}\end{equation}
A direction for solving this problem  is given by (\ref{bb}): we have that $t \mathfrak{b}_t\limdis \mathfrak{e}$ as $t\to \infty$. Take in Corollary \ref{XYZ}
$$(U_t,V_t,W_t)=(t \, \mathfrak{b}_t,\frac{X_{(t)}}{\rho_t}, Z_t), \quad \mbox{with}\; \rho_t= t\,c_t$$ and assume $\er[Z_\infty^{\lao}]$ is finite for some $\lao\in \mathds{T}\rsetminus\{0\}$. Under the last assumption, it could be noticed that problem (\ref{pb})  is equivalent to finding  necessary and  sufficient conditions on the deterministic and positive  normalization speed $\rho_t$, such that
\begin{equation}
X_t :=\frac{X_{(t)}}{\rho_t} \; \limdis \;X_\infty\quad \mbox{when $t\in \TT$ and}\; t \to \infty,
\label{pb1}\end{equation}
and such that $\er[X_\infty^{\lao}]$ is finite for some $\lao\in \mathds{T}\rsetminus\{0\}$. The random variable  $Z_\infty$ in (\ref{pb}) is then linked to $X_\infty$  by
\begin{equation} Z_\infty\simdis \mathfrak{e}\,X_\infty, \quad \mbox{where $\mathfrak{e}$ is exponentially distributed, independent from $X_\infty$.}
\label{zex} \end{equation}
Theorems \ref{th3} and \ref{th4} below, improve  the discrete time problem in (\ref{pb}, case $\TT=\mathds{N}$) studied by Harkness and Shantaram \cite{hs1}, by giving  a sharper answer through the continuous time problem  in (\ref{pb}, \ref{pb1}, case $\TT=[0,\infty)$).
\begin{theorem}[A normal limit theorem] Let  ${(X_t)}_{t\in \TT}$  the family  given by (\ref{pb1}).

\noindent 1) Assertions (i)-(ii)-(iii) are equivalent as $t\to \infty$:

(i)  $X_t$ converges in distribution to a non-null and nonnegative random variable  $X_\infty$  and
$$\lim_{t\to \infty}\er[X_t^{\lao}]=\er[X_\infty^{\lao}]<\infty, \quad \mbox{for some}\;\lao\in \TT \rsetminus\{0\};$$

(ii) $X_t$ converges in distribution to a non-null and nonnegative random variable $X_\infty$  and
$$\limsup_{t\to \infty}\frac{\rho_{t+s}}{\rho_t}<\infty, \quad \mbox{for some}\; s\in \TT \rsetminus\{0\};$$

(iii) there exists a  non-null and nonnegative random variable $X_\infty$ such that $[0,\infty) \in \mathcal{D}_{X_\infty}$ and
$$\er[X_t^\lambda]\to\er[X_\infty^\lambda], \quad \mbox{for all}\;\lambda \in [0,\infty).$$
\noindent 2) In this case, necessarily, $\rho_t \siminfini \er[X_\infty] \er[X^{t+1}]/\er[X^t]$ and there exists $c\geq 0$ such that
\begin{equation}
 l(s):=\lim_{t\to \infty}\frac{\rho_{t+s}}{\rho_t}= e^{c s},\quad \mbox{for every $s\in \TT$}.
 \label{rts} \end{equation}
3) Assume one of the equivalent assertions in 1) and let $c$ given by (\ref{rts}). Choose  $a:=\log \er[X_\infty]$ and let $N$ a random variable  normally distribution with mean $a-\frac{c}{2}$ and variance $c$ (it is understood that $N=a$ when $c=0$). Then the following holds:

(i) if $\,\TT=\mathds{N}$, then the law of the random variable $X_\infty$ is not determined by its integer moments, we only have
\begin{equation}
\er[X_\infty^\lambda]=\er[e^{\lambda N}]=e^{(a-\frac{c}{2})\lambda + \frac{c}{2}\lambda^2}, \quad \mbox{\it for all} \;\lambda \in \mathds{N};
 \label{normal} \end{equation}

(ii)  if $\,\TT=[0,\infty),\;$  then  $\;\log X_\infty \simdis N$, i.e. (\ref{normal}) holds for all $\lambda \in [0,\infty)$.

(iii) We have the identity in law
\begin{equation}
(X_\infty)_{(s)}\simdis  e^{cs} X_\infty,\quad \mbox{\it for every} \;s\in \TT.
\label{xs}\end{equation}
\label{th3}\end{theorem}
\begin{proof} Using properties (P1) and (P2), we start by noticing the following identity valid for every $t,\,s,\,\mu  \in \TT$ and $x\geq 0$:
$$\er[ X_t^{s +\mu}\,{\II}_{ X_t> x}]= \er\big[\big(\frac{X_{(t)}}{\rho_t}\big)^\lambda\;{\II}_{X_{(t)}> x \rho_t}\big] =\frac{\er[X^{t+s +\mu} \;{\II}_{X> x \rho_t}]}{\rho_t^\lambda\; \er[X^t]}= \frac{\er[X^{t+s}]}{\rho_t^\lambda\; \er[X^t]} \;\; \er[X_{(t+s)}^\mu  \;{\II}_{X_{(t+s)}> x \rho_t}].$$
In particular, for every $t,\,s,\,\mu  \in \TT$ and $x,\,y\geq 0$, we have
\begin{eqnarray}
\er[ X_t^{s +\mu}\,{\II}_{ X_t> x}]&=& \er[X_t^s]\; \left(\frac{\rho_{t+s}}{\rho_t}\right)^\mu  \;  \er[X_{t+s}^\mu  \;{\II}_{X_{t+s}> x \frac{\rho_t}{\rho_{t+s}}}]\label{e1}\\
\er[ X_t^{s}\,{\II}_{ X_t>  x}]&=& \er[X_t^s]\; \pr(X_{t+s}> x \frac{\rho_t}{\rho_{t+s}}) \label{e2}\\
\er[ X_t^{s}\,{\II}_{ X_t\leq  y}]&=&\er[X_t^s]\; \pr(X_{t+s}\leq  y \frac{\rho_t}{\rho_{t+s}})\label{e3}
\end{eqnarray}
If $X_t$ converges in distribution to $X_\infty$ and if  $z$ is a continuity point of $u\mapsto \pr (X_\infty \leq Ku)$, then for every $\epsilon \in (0,1)$, there exists $t_z\in \TT$ such that
\begin{equation}
|\pr (X_\infty \leq K z) -\pr (X_{t+s} \leq z)|=|\pr (X_{t+s} > z)-\pr (X_\infty > K z)| <\epsilon,\quad \mbox{for all}\;t\geq t_z.\label{14}\end{equation}
1)\underline{$(iii)\Longrightarrow(ii)$}  is easy, since by (\ref{e1}) with $x=0$, we have
\begin{equation}\er[ X_\infty^{s +\mu}]=\lim_{t\to \infty }\er[ X_t^{s +\mu}]= \er[X_\infty^s]\;\er[X_\infty^\mu]\; \lim_{t\to \infty} \left(\frac{\rho_{t+s}}{\rho_t}\right)^\mu\quad \mbox{for every} \;s,\,\mu \in \TT\rsetminus \{0\} .
\label{15}\end{equation}
\noindent\underline{$(ii)\Longrightarrow(iii)$}  Here $s\in \TT\rsetminus \{0\}$ is fixed and we proceed through two steps:

\noindent \underline{{\it Step 1}}:  we know that there exits $K>0, \,t_{_s}\in \TT$   such that
\begin{equation} \rho_t/\rho_{t+s} \geq K, \quad \mbox{for}\;t\geq t_{_s}.
\label{13}\end{equation}
The last inequality combined  with identities (\ref{e2}) and (\ref{e3}) give that for every $t\geq t_{_s},\,x,\,y>0$,
\begin{equation}\er[ X_t^{s}\,{\II}_{ X_t>  x}]\leq  \er[X_t^s]\; \pr(X_{t+s}> K x )  \quad \mbox{and}\quad y^s  \geq \er[X_t^s]\;   \pr (X_{t+s} \leq Ky).\label{e4}
\end{equation}
Now we choose $\epsilon =1/4$  in  (\ref{14})  with a continuity point $z=y_{_0}$ such that $\pr(X_\infty \leq K y_{_0})>1/2$. We get by the second inequality in (\ref{e4}) that
$$y_{_0}^s\geq \er[ X_t^{s}\,{\II}_{ X_t\leq  y_{_0}}]=\er[X_t^s]\;   \left(\pr (X_\infty \leq Ky_{_0})-\frac{1}{4}\right) > \frac{1}{4} \er[X_t^s],\quad t\geq t_{y_{_0}}.$$
We deduce that $C(y_{_0}):=\sup_{t\geq t_{y_{_0}}}\er[X_t^s]<\infty$. The first inequality in (\ref{e4}) gives
$$\er[ X_t^{s}\,{\II}_{ X_t> x}] \leq  C(y_{_0}) \;   \pr(X_{t+s}> x K), \quad t\geq t_{y_{_0}}, \, x>0.$$
The next step is to choose  $\epsilon$ arbitrary small in  (\ref{14}) with a continuity point   $z=x_{_0}$, big enough, so that $\pr(X_\infty > K x_{_0})< \epsilon$ in order to have for  $t\geq \max(t_{_s}\,,\,t_{x_{_0}}\,,\,t_{y_{_0}})$ and $x\geq x_{_0}$, that
$$\er[ X_t^{s}\,{\II}_{ X_t> x}]\leq \er[ X_t^{s}\,{\II}_{ X_t> x_{_0}}] \leq  C(y_{_0})\;\pr(X_{t+s}> x K)\leq  C(y_{_0})\;\pr(X_\infty> x_{_0} K)+\epsilon \leq (1+C(y_{_0}))\epsilon.$$
The latter justifies that ${(X_t)}_{t\in \TT}$ is $s$-uniformly integrable.

\noindent \underline{{\it Step 2}}: By (\ref{e1}) and (\ref{13}), we have for every $t\geq t_{_s}$,
$$\er[ X_t^{2s}\,{\II}_{ X_t> x}]= \er[X_t^s]\; \left(\frac{\rho_{t+s}}{\rho_t}\right)^s  \;  \er[X_{t+s}^s  \;{\II}_{X_{t+s}> x \frac{\rho_t}{\rho_{t+s}}}]\leq \er[X_t^s]\; \left(\frac{\rho_{t+s}}{\rho_t}\right)^s  \;  \er[X_{t+s}^s  \;{\II}_{X_{t+s}> K x}] . $$
In step 1, we gained that the family $M(s)=\limsup_{t\in \TT} \er[X_t^s]<\infty$  that ${(X_t)}_{t\in \TT}$ is $s-UI$. The last inequality shows that ${(X_t)}_{t\in \TT}$ is also $2s-UI$. Repeating the procedure, we obtain that ${(X_t)}_{t\in \TT}$ is also $ms-UI$ for every positive integer $m$ and    then, by Proposition \ref{prop0},  ${(X_t)}_{t\in \TT}$ is also $\lambda-UI$ for every positive number $\lambda >0$. Then we apply Theorem \ref{th2}.

\noindent\underline{$(iii)\Longrightarrow(i)$}  is an immediate application of point 1) in Theorem \ref{th2}.

\noindent \underline{$(i)\Longrightarrow(iii)$}: Identity (\ref{e1}) shows that
$$\er[ (X_t)_{(\lao)}^\mu] =  \left(\frac{\rho_{t+\lao}}{\rho_t}\right)^\mu  \;  \er[X_{t+\lao}^\mu], \quad  \mbox{for every}\; t,\, \mu  \in \TT.$$
By Lemma \ref{mellin}, we deduce that
$$X_{t+\lao} \simdis \frac{\rho_t}{\rho_{t+\lao}}\; (X_t)_{(\lao)}.$$
By Proposition \ref{melib}, we have that $(X_t)_{(\lao)}$ converges in distribution and that the triplet   $(U_t,V_t,W_t)= (X_{t+\lao},(X_t)_{(\lao)},\rho_{t+\lao}/\rho_t)$ satisfies Corollary \ref{XYZ}. We obtain that $\rho_{t+\lao}/\rho_t$ converges as $t \to \infty$ and then we use  $(ii)\Longrightarrow(iii)$.

2) The first claim stems from $\er[X_\infty]= \lim_{t\to \infty} \er[X_{(t)}]/\rho_t$. For the second claim, notice by Corollary \ref{nef}, that the function $t\mapsto \rho_t$ is asymptotically increasing, so that $l(s)\geq 1$ for every $s \in \TT$.  By (\ref{15}), we recover that
\begin{equation}
l(s)^\mu=\lim_{t\to \infty} \left(\frac{\rho_{t+s}}{\rho_t}\right)^\mu = \frac{\er[ X_\infty^{s +\mu}]}{ \er[X_\infty^s]\;   \er[X_\infty^\mu]} , \quad \mbox{for every}\;   s,\,\mu  \in \TT\rsetminus \{0\}.
\label{lsm}\end{equation}
From the symmetry in (\ref{lsm}), it is seen that $l(s)^\mu= l(\mu)^s$ for every $s,\, \mu \in \TT$. Taking  $c=\log l(1)\geq 0$, we get representation (\ref{rts}). The latter could be also deduced from Lemma 1 in \cite{aj}.

3) By Proposition \ref{melib},  $(X_t)_{(s)}  \limdis (X_\infty)_{(s)}$ for all $s\in \TT$, as $t\to \infty$,  and by properties (P0) and (P3), we obtain
$$(X_t)_{(s)}\simdis\frac{\rho_{t+s}}{\rho_t} \, X_{(t+s)} \limdis e^{c\mu} \, X_\infty \simdis (X_\infty)_{(s)}.$$
The latter  gives that the Mellin transform $\lambda \mapsto \mathcal{M}_{X_\infty}(\lambda)=\er[X_\infty^\lambda]$ is a solution of the functional equation:
\begin{equation}
h(1)=\er[X_\infty]=e^{a}, \qquad h(s +\mu)= e^{cs\mu} h(s)\;   h(\mu)  , \quad \mbox{for every}\; s,\,\mu  \in \TT,
\label{itera}\end{equation}
and this could be also  from identity (\ref{lsm}). The function $h_0(\lambda)= e^{c\lambda(\lambda -1)}$ solves the equation without the initial condition. Any solution of the form $h=h_0\, k$, has necessary $k(s +\mu)= k(s)\;   k(\mu)$ for every $s,\,\mu  \in \TT$, so that $k(\lambda)=k(1)^\lambda$. Due the initial condition, necessarily $k(1)=e^{a}$.
\end{proof}

If $\TT=\mathds{N}$, then identity (\ref{xs}) true for every $s\in \mathds{N}$ is equivalent to the same identity with $s=1$:
\begin{equation}\label{ind}
(X_\infty)_{(1)}\simdis  e^{c} X_\infty.
\end{equation}

We stress that  identity (\ref{ind}) does not allow to recover the log-normal distribution of $X_\infty$ which is not moment determinate. This situation was studied by many authors,  \cite{agk, bb, gar, hs2, pak1, pak2, vard} for instance and all these works were motivated by finding the set or possible limit for the discrete problem (\ref{pb0}).  We also stress that Harkness and Shantaram \cite{hs1} only showed, and in case $\TT=\mathds{N}$, that condition $(iii)$ in our Theorem \ref{th3}, implies $(i)$ and $(ii)$ without specifying the distribution of $Z_\infty$ which we know equal in distribution to $\mathfrak{e} \, X_\infty$. Theorem \ref{th3} distinguishes between the situation $\TT=\mathds{N}$ and $\TT= [0,\infty)$.  It is trivial that convergence of the family ${(X_t)}_{t\in [0,\infty)}$ implies that the subsequence ${(X_t)}_{t\in \mathds{N}}$ converges to the same limit. Theorem \ref{th4} below shows that the converse is true and  that actually in both discrete and continuous time problems, the only possible limits of normalized biased families are the log-normal distributions. Equivalently, the only possible limits of normalized families obtained by the stationary excess operator are the mixture of the exponential and the log-normal distribution.
\begin{theorem}[The normal limit theorem improved]  The following statements are equivalent:

\noindent (i)  convergence {\rm (\ref{pb}, case $T=\mathds{N}$)}
holds and $\mathcal{D}_{Z_\infty}$ contains some value  $\lao  \in
(0,\infty)$;

\noindent (ii)  convergence {\rm (\ref{pb}, case $T=[0,\infty)$)}
holds and $\mathcal{D}_{Z_\infty}$ contains some value $\lao
\in(0,\infty)$;

\noindent (iii)  convergence {\rm (\ref{pb1}, case $T=\mathds{N}$)}
holds and $\mathcal{D}_{X_\infty}$ contains some value $\lao
\in(0,\infty)$;

\noindent (iv)  convergence {\rm (\ref{pb1}, case $T=[0,\infty)$)}
holds and $\mathcal{D}_{X_\infty}$ contains some value $\lao
\in(0,\infty)$.

\noindent In all cases,  $\mathcal{D}_{Z_\infty}$ and $\mathcal{D}_{X_\infty}$ necessarily contain $[0,\infty)$ and convergence (\ref{rts}) holds with some $c\geq 0$.  We also have   $Z_\infty \simdis \mathfrak{e} X_\infty$ where $\mathfrak{e}$  and $X_\infty$ are independent and have respectively the standard exponential distribution and the log-normal distribution, i.e., if for every choice of  $\alpha =\er[X_\infty]$, the random variable  $\log X_\infty$ has the normal distribution with mean equal to $\log \alpha- \frac{c}{2}$ and variance equal to  $\frac{c}{2}$. It is understood that $X_\infty = \alpha\,$ if  $\,c=0$. Furthermore we have the identity in law
$$Z_\infty \simdis e^{-cs} \;\mathfrak{b}_s\;(Z_\infty)_{(s)},\quad \mbox{\it for every}\; s\geq 0 $$
where $\mathfrak{b}_s$ is assumed to be independent from $(Z_\infty)_{(s)}$.
\label{th4}\end{theorem}
\begin{proof} By the discussion before Theorem \ref{th3}, we know that $(i)\Longleftrightarrow(iii)$ and that $(ii)\Longleftrightarrow(iv)$.  $(iv)\Longrightarrow(iii)$ being trivial, it remains to show $(iii)\Longrightarrow(iv)$. By Theorem \ref{th3}, it is enough to show that
\begin{equation}
\limsup_{n\in \mathds{N},\;n\to \infty} \frac{\rho_{n+1}}{\rho_n} <\infty \quad \mbox{is equivalent to}\quad    \limsup_{t\in [0,\infty),\;t\to \infty} \frac{\rho_{t+s}}{\rho_t} <\infty, \quad  \mbox{for all}\, s>0.
\label{fin}\end{equation}
By Theorem \ref{th3}, we also know that
$$\rho_t \siminfini r (t)= \er[X_\infty]\, \er[X_{(t)}]$$
and  Corollary  \ref{nef}  says that the function  $t\mapsto r(t)$  is nondecreasing.  Let $[x]$ the integer of the real number $x$. We have $[t] \leq t \leq t+s\leq [t]+[s]+2$ for every $t,\,s > 0$ and then $r(t+s)/r(t)\leq r([t]+[s]+2)/r([t])$. It is then immediate that
\begin{eqnarray*}
\limsup_{t\in [0,\infty),\;t\to \infty}\frac{\rho_{t+s}}{\rho_t} &=&\limsup_{t\in [0,\infty),\;t\to \infty} \frac{r(t+s)}{r(t)}\leq  \limsup_{t\in [0,\infty),\;t\to \infty} \frac{r([t]+[s]+2)}{r([t]}= \limsup_{n\in \mathds{N},\;n\to \infty} \frac{r(n+[s]+2)}{r(n)}\\
 &\leq&  \limsup_{n\in \mathds{N},\;n\to \infty} \left(\frac{r(n+1)}{r(n)}\right)^{[s]+2)}= \limsup_{n\in \mathds{N},\;n\to \infty} \left(\frac{\rho_{n+1}}{\rho_n}\right)^{[s]+2)}<\infty.
\end{eqnarray*}
On the other hand,
$$\limsup_{t\in [0,\infty),\;t\to \infty}\frac{\rho_{t+1}}{\rho_t}\geq  \limsup_{n\in \mathds{N},\;n\to \infty}\frac{\rho_{n+1}}{\rho_n},$$
which shows (\ref{fin}). Last identity in the theorem stems from property (P4), identities (\ref{zex}), (\ref{bia}) and then Beta-Gamma algebra identities (\ref{algb}):
$$(Z_\infty)_{(s)}  \simdis (\mathfrak{e}\, X_\infty)_{(s)}\simdis \mathfrak{g}_{s+1}\, (X_\infty)_{(s)}$$
which yields
$$e^{-cs} \,\mathfrak{b}_s\,(Z_\infty)_{(s)}  \simdis \mathfrak{e}  X_\infty  \simdis Z_\infty .$$
\end{proof}

In order to provide an example, we recall the following concept: Let  $g:(0,\infty)\to \mathds{R}$ and $\Delta_a$ the difference operator given by $\Delta_a g(x) := g(x+a)-g(x)$. The function $g$ is   said monotone of order $k\in \mathds{N}$, if it satisfies
$$(-1)^k\, \Delta_{a_1}\Delta_{a_2}\cdots \Delta_{a_k} g \leq 0,\qquad \mbox{for all} \,\; a_1,a_2,\cdots a_k>0, \; k\in \mathds{N}\rsetminus \{0\}.$$
Assume $g$ is $k$ times  differentiable. The discussion in \cite{SSV} at the end of page 43 indicates, by a  mean value theorem argument, that for every $x\in (0,\infty)$,
$$ \Delta_{a_1}\Delta_{a_2}\cdots \Delta_{a_k} g(x)= g^{(k)}(x+\theta_1 a_1+ \theta_2a_2+\cdots \theta_k),$$
for suitable $\theta_1 , \theta_2,\cdots, \theta_k\in (0,1)$. It is then immediate that $(-1)^k\, g^{(k)}\geq 0$ implies that $g$ is monotone of order $k$.
By Theorem 4.11 p.42 \cite{SSV},  $g$  is $n$-monotone is equivalent to its nonnegativity and monotonicity of all order $k=1, \cdots n$ and is also equivalent to $(-1)^k\, g^{(k)}\geq 0$ for all $k=0, \cdots n$.

We remind that $X$  is a nonnegative random variable such that $[0,\infty)\subset \mathcal{D}_X$. Now choose a positive  quantity $\alpha$, to be allocated to the value of $\er[X_\infty]$, and define for $t\geq 0$
$$ g_X(t)=\log \er[X^t]\quad \mbox{and} \quad   \rho_t = \alpha \, \er[X_{(t)}]= \alpha \frac{\er[X^{t+1}]}{\er[X^t]}=\alpha  \exp  \Delta_1 g_X (t).$$
We already know that by Proposition \ref{lcx} that $g_X$ is  monotone of order 2 (that is $g_X$ is convex) and by Corollary \ref{nef}, we know that   $t \mapsto\rho_t$ in nondecreasing and then the quantity
\begin{equation}\label{1s}
\frac{\rho_{t+s}}{\rho_t} = \frac{\er[X_{(t+s)}]}{\er[X_{(t)}]}= \exp  \Delta_1\Delta_s g_X (t)
\end{equation}
is bounded below by $1$. It is clear that if $g_X$   is   monotone of order 3 (this holds if the derivative $g_X'$ is concave),  then the function   $t\mapsto  \rho_{t+s}/\rho_t$  is nonincreasing. The latter implies that $\lim_{t\to \infty} \rho_{t+s}/\rho_t$  exits  and  is necessarily as in (\ref{rts}).  We are now able to provide examples of random variables satisfying (\ref{rts}). For infinite divisibility property  of real random variables, the reader is referred to the book of Steutel-van Harn \cite{steutel} and also to \cite{SSV}:
\begin{example}{\rm  If  $X$ is a random variable such  that  $g_X'$ is a concave function, then (\ref{rts}) is satisfied. For instance, assume $\log X$ is an infinite divisible random variable such that its L\'evy exponent $g_X=\log \mathcal{M}_X$  has the form
\begin{equation}\label{logx}
g_X(\lambda)= d\lambda + \frac{\sigma^2}{2} \lambda^2 +\int_{(0,\infty)}(e^{-\lambda x} -1 +\lambda x\II_{ x \leq 1})\, \pi(dx),\quad \lambda \geq 0,
\end{equation}
with $d \in \mathds{R}, \, \sigma\geq 0$ and the L\'evy measure $\pi$ satisfy the $\int_{(0,\infty)}(x^2\wedge 1) \pi(dx)<\infty.$ It  is easy to check that $g_X'$ is concave. Furthermore, we have
$$ \Delta_1\Delta_s g (t) = \sigma^2 s + \int_{(0,\infty)}e^{-t x}(1-e^{- x})(1-e^{-s x})\, \pi(dx),\quad t,\,s>0,$$
and by (\ref{1s}), $X$ satisfies (\ref{rts}) with $c=\sigma^2.$}
\end{example}


\end{document}